\documentclass[12pt, reqno]{amsart}
\usepackage{amsmath, amsthm, amscd, amsfonts, amssymb, graphicx, color}
\usepackage[bookmarksnumbered, colorlinks, plainpages]{hyperref}
\hypersetup{colorlinks=true,linkcolor=red, anchorcolor=green, citecolor=cyan, urlcolor=red, filecolor=magenta, pdftoolbar=true}

\usepackage{epsfig, enumerate}
\DeclareMathAlphabet{\mathpzc}{OT1}{pzc}{m}{it}

\newtheorem{theorem}{Theorem}[section]
\newtheorem{lemma}[theorem]{Lemma}
\newtheorem{proposition}[theorem]{Proposition}
\newtheorem{corollary}[theorem]{Corollary}
\theoremstyle{definition}

\newtheorem{example}[theorem]{Example}

\theoremstyle{remark}
\newtheorem{remark}[theorem]{Remark}
\numberwithin{equation}{section}

\def\q{\mathpzc q}

\def\eps{\varepsilon}

\def\disp{\displaystyle}
\def\p{\mathcal{P}}

\def\cd{\mathrm{cd}}

\def\HB{\textit{HB}}

\begin{document}
\setcounter{page}{1}

\title{Ideal structures in vector-valued polynomial spaces}

\author[V. Dimant, S. Lassalle, A. Prieto]{Ver\'{o}nica Dimant$^1$, Silvia
Lassalle$^1$, \'{A}ngeles Prieto$^2$}

\address{$^{1}$Departamento de Matem\'{a}tica, Universidad de San
Andr\'{e}s, Vito Dumas 284, (B1644BID) Victoria, Buenos Aires,
Argentina and CONICET.} \email{\textcolor[rgb]{0.00,0.00,0.84}{vero@udesa.edu.ar; slassalle@udesa.edu.ar}}

\address{$^{2}$Departamento de An\'{a}lisis Matem\'{a}tico, Facultad de CC.
Matem\'{a}ticas, Universidad Complutense de Madrid,
Plaza de Ciencias, 3, 28040 Madrid, Spain.} \email{\textcolor[rgb]{0.00,0.00,0.84}{angelin@mat.ucm.es}}


\begin{abstract}
This paper is concerned with the study of geometric structures in spaces of
polynomials. More precisely, we discuss for $E$ and $F$ Banach spaces, whether
the class of  $n$-homogeneous polynomials,
$\mathcal P_w(^n E, F)$, which are weakly continuous on bounded sets, is an \HB-subspace or an $M(1,C)$-ideal in the space of
continuous $n$-homogeneous polynomials, $\mathcal P(^n E, F)$. We establish
sufficient conditions under which the problem can be positively solved. Some
examples are given. We also study when some ideal structures pass from $\mathcal
P_w(^n E, F)$ as an ideal in $\mathcal P(^n E, F)$  to the range space $F$ as an
ideal in its bidual $F^{**}$.
\end{abstract}

\maketitle

\section{Introduction}
Let $X$ be a (real or complex) Banach space and let $J$ be a closed subspace of
$X$. According to the Hahn--Banach theorem, every continuous linear functional
$g \in J^*$ has an extension $f\in X^*$, with the same norm. A long standing
problem is to determine when every functional on $J$ has a \textbf{unique}
norm-preserving extension to $X$. This question is closely related to
geometric properties of both spaces which, in many cases, imply the existence of
a norm-one projection on $X^*$ whose kernel is $J^{\perp}:=\{x^*\in X^*\colon
x^*(y)=0,\, \text{for all } y \in J \}$, the annihilator of $J$. When there
exists such a projection $J$ is said to be an \textit{ideal} in $X$. A canonical
example of this fact is that $X$ is always an ideal in its bidual $X^{**}$.

The notion of $M$-\textit{ideal}, introduced by Alfsen and Effros
and widely studied in the book by Harmand, Werner and Werner \cite{HWW} is one
of these geometric properties ensuring unique Hahn--Banach extensions. Recall
that $J$ \textit{is an $M$-ideal in} $X$ if it is an ideal in $X$ with
associated projection $\q$ such that for each $f\in X^*$ one has
$$
\|f\|=\|\q f\| + \|f-\q f\|.
$$
The fact that $J$ is an $M$-ideal in $X$ has a strong impact on both $J$ and
$X$, and sometimes seems to be too restrictive. So, we will be interested in
studying some weaker properties among those  implying unique norm-preserving
extensions.

Recall that a closed subspace $J$ is \HB-\textit{smooth} in $X$ if  every
element in $J^*$ has a unique norm-preserving extension to an element in $X^*$.
A closed subspace $J$ is \textit{strongly} \HB-\textit{smooth} in $X$ if there
exists a linear projection $\q$ on $X^*$ whose kernel is $J^\perp$ such that for
each $f\in X^*$ with $f\ne \q f$ one has
$$
\|\q f\|<\|f\|.
$$
The interplay between uniqueness of the extension and strong \HB-smoothness
was clarified by Oja \cite{Oja-88}. Namely, the uniqueness of the extensions
and being an ideal are independent notions for a subspace $J$, strong
\HB-smoothness implies both, and if $J$ is an \HB-smooth ideal in $X$ then $J$
is strongly \HB-smooth in $X$.

A particular case of \HB-smoothness is the notion of \HB-\textit{subspace},
introduced by Hennefeld \cite{Hen}. A closed subspace $J$ \textit{is an
\HB-subspace of} $X$ if there exists a projection $\q$ on $X^*$ whose kernel is
$J^\perp$ such that for each $f\in X^*$ with $f\ne \q f$ one has
$$
\|\q f\|<\|f\| \qquad \text{and} \qquad \|f-\q f\|\leq \|f\|.
$$

Finally, given $C \in (0,1]$, a closed  subspace $J$ is an
$M(1,C)$-\textit{ideal} in $X$ if $J$ is an ideal of $X$ with associated
projection $\q$ on $X^*$ such that for each $f\in X^*$ one has
$$
\|\q f\|+C\|f-\q f\|\leq \|f\|.
$$
The last inequality is called the $M(1,C)$-inequality. Note that when $C=1$,
the notion of $M$-ideal is covered and to be $M(1,C)$-ideal immediately implies
strong \HB-smoothness. However, the notions of $M(1,C)$-ideal and \HB-subspace
are independent. On the one hand, Cabello and Nieto \cite[Example~3.7]{CabNie2}
showed that if $X$ is a nonreflexive separable $M$-ideal in its bidual, then
$\ell_p(X)$ as a subspace of its bidual, $1<p<\infty$, is an \HB-subspace that
cannot be renormed to be an $M(1,C)$-ideal for any $0<C<1$. On the other hand,
Cabello, Nieto and Oja \cite[Example~4.3]{CabNieOja} showed that for any
$0<C<1$, there is a renorming of $c_0$, $\hat c_0$ due to Johnson and Wolfe,
such that the space of compact operators on $\hat c_0$, $\mathcal K(\hat c_0)$
is an $M(1,C)$-ideal in the space of all continuous operators $\mathcal L(\hat
c_0)$ without being an \HB-subspace.

Several authors have been interested in this kind of properties for arbitrary
subspaces of Banach spaces, and also for distinguished particular cases. The
space of compact operators $\mathcal K(E,F)$ between Banach spaces $E$ and $F$
as a subspace of the space of all continuous linear operators $\mathcal L(E,F)$
received special interest (see, for example,
\cite{CabNie2,CabNieOja,Hen,LL,Oja-88,Oja-89,Oja-Pol,W}). The strongest of the
abovementioned properties is the one of being an $M$-ideal. All other properties which are
more flexible still allow us to deal with uniqueness of Hahn--Banach extensions.

Here, we will be concerned with $\mathcal P(^n E, F)$, the space of continuous
$n$-homo\-gen\-eous polynomials between Banach spaces $E$ and $F$. In the polynomial
context, the space of \textit{compact mappings} is usually replaced by $\mathcal
P_w(^nE, F)$ the subspace of \textit{homogeneous polynomials which are weakly
continuous on bounded sets}. Recall that a polynomial $P\in \mathcal  P(^n E,
F)$  is in $\mathcal P_w(^nE, F)$ if it  maps bounded weakly convergent nets
into convergent nets. Note that we could have considered polynomials in $\p(^n
E, F)$ mapping bounded sets into relatively compact sets, which are called
compact polynomials. For linear operators to be compact and to be weakly
continuous on bounded sets are equivalent notions. For $n$-homogeneous
polynomials with $n>1$, every polynomial in  $\mathcal P_w(^nE, F)$ is compact
(as can be derived from results in \cite{AHV} and \cite{AP}), but the converse
might not be true. Every scalar-valued continuous polynomial is compact but it
is not necessarily weakly continuous on bounded sets. The prototypical example
of this situation is given by $P(x)=\sum_k x_k^2$, for all $x=(x_k)_k\in
\ell_2$. Therefore, we will focus our attention on determining the presence of
ideal structures for $\mathcal P_w(^nE, F)$ as a subspace of $\mathcal P(^n E,
F)$. To be more precise, our main concern is to study the notion of \HB-subspace
in the polynomial setting.

Some previous results in this direction can be found in \cite{Dim},
where the problem of determining when  $\mathcal P_w(^nE)$ is an $M$-ideal in
$\mathcal P(^nE)$ was considered. A vector-valued approach of the same question
was treated in  \cite{DiLa}. Note that the searching of ideal structures for
$\mathcal P_w(^nE, F)$ as a subspace of $\mathcal P(^n E, F)$ makes sense when
the spaces $\mathcal P_w(^nE, F)$ and $\mathcal P(^n E, F)$ do not coincide. The
equality $\mathcal P_w(^nE, F)=\mathcal P(^n E, F)$ is a long standing
nontrivial problem, considered for instance in \cite{AF,BR-98,GG,
GJ}.

The plan of the paper is as follows. First, we review the notation and the
basic facts that will be used in Sections 3 and 4. Then, in Section 3, we investigate
sufficient conditions under which the subspace $\mathcal P_w(^nE,F)$ enjoys an
additional geometric structure inside $\mathcal P(^nE,F)$ and we exhibit some
particular examples. In the last section, we study some ideal structures for the
range space $F$ as a subspace of $F^{**}$, when they are fulfilled by $\mathcal
P_w(^nE,F)$ as a subspace of $\mathcal P(^nE,F)$.
\bigskip

\section{Notation and basic facts}

Before proceeding, we fix some notation. Every time we write $E$  or $F$ we
will be considering Banach spaces over the real or complex field, $\mathbb{K}$.
The closed unit ball of $E$ will be denoted by $B_E$ and the unit sphere by
$S_E$. As usual, $E^*$ and $E^{**}$ stand for the dual and bidual of $E$,
respectively. The space of linear bounded operators from $E$ to $F$ will be
denoted by $\mathcal{L}(E,F)$ (and $\mathcal{L}(E)$ when $E=F$); its subspace of
compact mappings will be denoted by $\mathcal{K}(E,F)$ ($\mathcal{K}(E)$ in the
case $E=F$).

A function $P\colon E\to F$ is an $n$-homogeneous polynomial if there exists a
(unique) symmetric $n$-linear form $A\colon\underbrace{E\times\cdots\times E}_n\to F$ such that
$$
P(x)=A(x,\dots,x),
$$
for all $x\in E$. The space of all continuous $n$-homogeneous polynomials from $E$ to $F$,
$\mathcal P(^nE,F)$, endowed with the supremum norm
$$
\|P\|=\sup\{\|P(x)\| \colon \, x\in B_E\},
$$
is a Banach space.

Every  polynomial $P$ in $\mathcal P(^n E;F)$ can be associated with a linear operator
in $\mathcal{L}(\widehat\otimes^{n,s}_{\pi_s}E ; F)$, where $\pi_s$ is the
symmetric  projective tensor norm. We will identify $P$ with its linearization
without further mention. Even though this identification preserves the norm,
there is no Hahn--Banach Theorem for homogeneous polynomials of degree $2$ or
greater. However, Aron and Berner \cite{AB} and Davie and Gamelin
\cite{DaGa} showed that for every $P\in {\mathcal P}(^nE,F)$ there is a
norm-preserving extension of $P$ to $\overline P\in \mathcal
P(^nE^{**},F^{**})$ such that $\overline{P}(x)=P(x)$ for all $x\in E$. The
construction of this {\it canonical extension} is based on the Arens extension
of the symmetric mapping $A$ associated to the polynomial $P$. To obtain the
Arens extension, we simply extend by weak-star continuity, one variable at a
time, the $n$ variables of $A$. This process depends on the order that the variables
are extended and the final result might not be a symmetric mapping. However, the
$n!$ possible extensions coincide on the diagonal and $\overline P$ is well
defined. For the particular case in which $P$ belongs to $\mathcal P_w(^nE,F)$,
the range of  $\overline{P}$ is also in $F$ (as can be derived from \cite{AHV}
and \cite[Proposition~2.5]{CaLa}). This fact will be used repeatedly in
Section~4.

In this paper, we will present several results in which at least one of the
spaces involved enjoys the metric compact approximation property. Recall that a
Banach space $E$ has the \textit{metric compact approximation property} if there
is a net of compact operators $(K_\alpha)$ on $E$ such that $K_\alpha \to {\rm
Id}_E$ pointwise and $\sup_\alpha \|K_\alpha\|\le 1$. Usually, the net
$(K_\alpha)$ is called a \textit{metric compact approximation of the identity}.
If in addition $K^*_\alpha \to {\rm Id}_{E^*}$ pointwise, the net $(K_\alpha)$
is called a \textit{shrinking} metric compact approximation of the identity. As usual,
$K^{\alpha}$ denotes the operator $\text{Id}_E-K_{\alpha}$. For dual spaces we
have the following intermediate property. We say that $E^*$ has a \textit{metric
compact approximation of the identity with adjoint operators} if there exists a
net $(K_\alpha)\subset \mathcal K(E)$ such that $K_\alpha^*$ converges to ${\rm
Id}_{E^*}$ pointwise and $\sup_\alpha \|K_\alpha\|\le 1$. These notions are
closely related with ideal structures on Banach spaces. For instance,
\cite[Theorem 1.1]{LL} asserts that the following conditions are equivalent:

\begin{enumerate}[\upshape (i)]
\item  $F$ has the metric compact approximation property
\item  $\mathcal K(E,F)$ is an ideal in $\mathcal L(E,F)$ for every Banach space $E$.
\end{enumerate}

So, it is natural to expect that the metric compact approximation property
shows up when describing $\mathcal P_w(^nE, F)$ as an ideal in $\mathcal P(^nE,
F)$.

One further ingredient will appear in our discussion. In \cite{Delp}, Delpech
obtained an appropriate connection between the moduli of asymptotic uniform
smoothness and convexity and weak sequential continuity of polynomials. In
\cite{DiGoJa}, Dimant, Gonzalo and Jaramillo followed his approach to obtain
results on compactness or weak-sequential continuity of multilinear mappings.
Here, we will impose restrictions on the growth of the moduli of the underlying
spaces $E$ or $F$ to ensure that $\mathcal P_w(^nE, F)$ enjoys an appropriate
property in $\mathcal P(^nE, F)$ (see Theorem~\ref{asymptotic} and
Proposition~\ref{useful}). Some definitions are in order.

For an infinite dimensional Banach space $E$, the \textit{modulus of asymptotic
pointwise smoothness} is defined for $\|x\|=1$ and $t>0$ by
$$
\overline{\rho}_E(t;x)=\inf_{\dim(E/H)<\infty} \sup_{h\in H, \|h\|\leq t}
\|x+h\|-1,
$$
and the \textit{modulus of asymptotic uniform smoothness} is defined for $t>0$ by
$$
\overline{\rho}_E(t)= \sup_{\|x\|=1}\overline{\rho}_E(t;x).
$$
The space $E$ is \textit{asymptotically uniformly smooth} if
$\lim_{t \to 0} \displaystyle\frac{\overline{\rho}_E(t)}{t}=0$.

For an infinite dimensional Banach space, the \textit{modulus of asymptotic
pointwise convexity} is defined for $\|x\|=1$ and $t>0$ by
$$
\overline{\delta}_E(t;x)= \sup_{\dim(E/H)<\infty} \inf_{h\in H, \|h\|\geq t}
\|x+h\|-1,
$$
and the \textit{modulus of asymptotic uniform convexity} is
defined for $t>0$ by
$$
\overline{\delta}_E(t)= \inf_{\|x\|=1}\overline{\delta}_E(t;x).
$$
The space $E$ is \textit{asymptotically uniformly convex} if
$\overline{\delta}_E(t)>0$, for every $0<t\leq 1$. Finally, $E$ has modulus of
\textit{asymptotic uniform convexity of power} $p$ if there exists $C>0$ such
that $\overline{\delta}_E(t) \geq C t^p$, for all $0<t\leq 1$.

We refer to \cite{Di} for the necessary background on polynomials on Banach
spaces.

\section{Sufficient conditions}

When working with polynomials, the lack of linearity provides, in many cases,
difficulties that can be overcome not without certain detours. The value of $n$
for which  $\mathcal P_w(^nE, F)$ has the chance to be a nontrivial $M$-ideal in
$\mathcal P(^nE, F)$ cannot be chosen arbitrarily. In fact, in the
scalar-valued case \cite{Dim} it was proved  that, whenever $\mathcal P(^mE)
\setminus \mathcal P_w(^mE) \neq \emptyset$ for some $m$, there exists a unique
value $n$, called the \textit{critical degree}, for which $\mathcal P_w(^nE)$
can be a non-trivial $M$-ideal in $\mathcal P(^nE)$. The critical degree of $E$
is defined as
$$
\cd(E):=\min\{k\in \mathbb N\colon \mathcal P_w(^kE) \neq \mathcal P(^kE)\}.
$$
In the vector-valued case, the critical degree is defined by analogy \cite{DiLa} as
$$
\cd(E,F):=\min\{k\in \mathbb N\colon \mathcal P_w(^kE, F) \neq \mathcal P(^kE,
F)\},
$$
and the problem of whether $\mathcal P_w(^nE,F)$ is an $M$-ideal in $\mathcal
P(^nE,F)$ is worth being studied only for polynomials of degree $n$ with
$\cd(E,F)\leq n \leq \cd(E)$. Although we are interested in studying ideal
structures which are more flexible than to be an $M$-ideal, in order to show
positive results we could not get rid of some restrictions on the degree of
homogeneity. We start with a lemma which, under certain conditions on $n$,
gives a version of Johnson's projection \cite [Lemma 1.1]{John} for the polynomial
case.

\begin{lemma}\label{idealF}
Let $E,F$ be Banach spaces and let $n< \cd(E)$. Suppose that $F$ has the metric
compact approximation property. Then $\mathcal P_w(^nE, F)$ is an ideal in
$\mathcal P(^nE, F)$.
\end{lemma}

\begin{proof}
Let $(L_{\beta}) \subset \mathcal K(F)$ be a metric compact approximation of
the identity. As $(L_{\beta})$ is bounded, there exists a subnet, still denoted
by  $(L_{\beta})$, that converges $w^*$ to some $L_0 \in \mathcal K(F)^{**}$.
Define $\Lambda\colon  \mathcal P(^nE, F)^* \to \mathcal P(^nE, F)^*$ by
\begin{equation}\label{projection}
\Lambda(f)(P)=\lim_{\beta}f(L_{\beta}\circ P).
\end{equation}
Note that $\Lambda$ is well defined. In fact, if $P \in \mathcal P(^nE, F)$
and $\tau_P\colon \mathcal K(F) \to \mathcal P(^nE, F)$ is  the composition
operator,   $\tau_P(K)= K\circ P$, its transpose $\tau_P^*$  satisfies that
$\tau_P^*(f) \in \mathcal K(F)^*$, for any $f \in  \mathcal P(^nE, F)^*$. So
$$
\lim_{\beta}f(L_{\beta}\circ P)=\lim_{\beta}\tau_P^*(f)(L_{\beta})=
L_0(\tau_P^*(f)).
$$
It is clear that $\Lambda$ is linear and $\|\Lambda \|\leq 1$. It is also a
projection: since any $P \in \mathcal P_w(^nE, F)$ is compact and $(L_\beta)$
converges to the identity on compact sets we see that $\lim_{\beta}L_{\beta}
\circ P=P$. Thus,
$$
\Lambda(f)(P)=\lim_{\beta}f(L_{\beta} \circ P)=f(P),
$$
for $P\in \mathcal P_w(^nE, F)$. Now, by \cite[Lemma 1.8]{DiLa}, as $n<
\cd(E)$, the net of polynomials $(L_{\beta}\circ Q)$ belongs to $\mathcal
P_w(^nE, F)$, for every $Q\in \mathcal P(^nE, F)$. Hence,
$$
\Lambda(\Lambda(f))(Q)=\lim_{\beta}\Lambda(f)(L_{\beta}\circ
Q)=\lim_{\beta}f(L_{\beta}\circ Q)=\Lambda(f)(Q),
$$
and $\Lambda^2=\Lambda$. Finally, it is easy to check that $\ker \Lambda=
\mathcal P_w(^nE, F)^{\perp}$. Then, $\Lambda$ is a norm one projection on
$\mathcal P(^nE, F)^*$ with $\ker \Lambda = \mathcal P_w(^nE, F)^{\perp}$.
\end{proof}

\begin{remark}\label{converge}
\rm Every time $\mathcal P_w(^nE, F)$ is an ideal
in $\mathcal P(^nE, F)$ with associated projection $\Lambda$, we have the
decomposition
$$
\mathcal P(^nE, F)^*=\mathcal P_w(^nE, F)^* \oplus \mathcal P_w(^nE, F)^{\perp},
$$
and any $f \in \mathcal P(^nE, F)^*$ has a unique representation such that
\begin{equation}\label{decomposition}
f=g+h,\quad \text{with}\quad g = \Lambda (f) \in \mathcal P_w(^nE, F)^*, \quad
h =f-g \in \mathcal P_w(^nE, F)^{\perp}.
\end{equation}

Now, if $F$ has a metric compact approximation of the identity  $(L_{\beta})
\subset \mathcal K(F)$, we may (and will) suppose that $(L_{\beta})$ is
$w^*$-convergent in $\mathcal K(F)^{**}$ and that the projection $\Lambda$ is
defined as in \eqref{projection}, $\Lambda(f)(P)=\lim_{\beta}f(L_{\beta}\circ
P)$. Then, with $L^{\beta}=\text{Id}_F-L_{\beta}$, we have the following facts
that were already used:
\begin{itemize}
\item $\lim_{\beta}\Lambda(f)(L^{\beta} \circ P)=0,$ for all $f \in \mathcal
P(^nE, F)^*$ and all $P \in \mathcal P(^nE, F)$.
\item $\lim_{\beta}L_{\beta} \circ Q=Q$, for all $Q \in \mathcal P_w(^nE, F)$.
\end{itemize}

Indeed, the first assertion follows by \cite[Lemma~1.8]{DiLa}. For the second
one, note that $L_\beta$ converges uniformly to the identity on compact sets.

Finally, note that since $\|\Lambda \|\le 1$, we automatically have $\|g\|\le
\|f\|$. If in addition $\|L^\beta\|\le 1$, we also obtain  $\|h\|\leq \|f \|$
since $h(P)=(f-\Lambda f)(P)=\lim_{\beta}f(L^{\beta}\circ P)$ for all $P \in
\mathcal P(^nE, F)$.
\end{remark}

As mentioned before, Delpech used the moduli of asymptotic uniform convexity
and smoothness of a Banach space to obtain properties of weak sequential
continuity of polynomials.  Dimant,  Gonzalo and Jaramillo \cite{DiGoJa} showed
the connection between these moduli and compactness or weak sequential
continuity of multilinear mappings. The moduli play their role when dealing with
$\mathcal P_w(^nE, F)$ as an \HB-subspace of $\mathcal P(^nE, F)$. We present a
refinement of \cite[Lemma~10.3]{Delp}, that will be used in the sequel.

\begin{lemma}\label{Delpech}
Let $E$ be an infinite dimensional Banach space and let $(w_{\alpha})\subset E$
be a weakly null bounded net.
\begin{enumerate}[\upshape (a)]
\item If $x \in B_E$, then $\,\overline{\lim}_{\alpha}\|x+w_{\alpha}\| \leq 1 +
\overline{\rho}_E(\overline{\lim}_{\alpha}\|w_{\alpha}\|)$.

\item If $(x_{\alpha})\subset B_E$ is a net contained in a compact set, then
$$
\overline{\lim}_{\alpha}\|x_{\alpha}+w_{\alpha}\|  \leq 1+
\overline{\rho}_E(\overline{\lim}_{\alpha}\|w_{\alpha}\|).
$$
\end{enumerate}
\end{lemma}

\begin{proof} To prove (a) first note that \cite[Lemma~10.3]{Delp} remains
valid if we consider weakly null nets instead of weakly null sequences. That is,
$\overline{\lim}_{\alpha}\|x+w_{\alpha}\| \leq 1 +
\overline{\rho}_E(\overline{\lim}_{\alpha}\|w_{\alpha}\|)$, for any $x\in S_E$.
Now, fix a nonzero $x\in B_E$ and consider each $x+w_{\alpha}$ as a convex
combination of $\displaystyle\frac{x}{\|x\|}+w_{\alpha}$ and
$\displaystyle\frac{-x}{\|x\|}+w_{\alpha}$. Applying the above inequality to
$\disp\frac{\pm x}{\|x\|}$ we get
$$
\overline{\lim}_{\alpha} \left\|\displaystyle\frac{\pm x}{\|x\|}+
w_{\alpha}\right\| \leq
1+\overline{\rho}_E(\overline{\lim}_{\alpha}\|w_{\alpha}\|),
$$
and the statement follows.

Now, suppose that (b) does not hold. Then, we may find subnets
$(x_{\beta}),(w_{\beta})$, and $x_0 \in B_E$, so that
$\lim_{\beta}x_{\beta}=x_0$ and
$\overline{\lim}_{\beta}\|x_{\beta}+w_{\beta}\|>1+\overline{\rho}_E(\overline{
\lim}_{\alpha}\|w_{\alpha}\|)$. As $\overline{\rho}_E$ is increasing,
$\overline{\lim}_{\beta}\|x_{\beta}+w_{\beta}\|>
1+\overline{\rho}_E(\overline{\lim}_{\beta}\|w_{\beta}\|)$. Note that for any
subnet $(\beta_i)$ such that $\lim_{i}\|x_{\beta_i}+w_{\beta_i}\|$ exists, so too does
the limit $\lim_{i}\|x_{0}+w_{\beta_i}\|$, and both coincide. This implies
that $\overline{\lim}_{\beta}\|x_{0}+w_{\beta}\|
=\overline{\lim}_{\beta}\|x_{\beta}+w_{\beta}\|$. It follows that
$\overline{\lim}_{\beta}\|x_{0}+w_{\beta}\|   >1 +
\overline{\rho}_E(\overline{\lim}_{\beta}\|w_{\beta}\|)$, which contradicts (a).
\end{proof}

We are ready to describe $\mathcal P_w(^nE, F)$ as an \HB-subspace of $\mathcal
P(^nE, F)$, under certain conditions on $F$ and $n$.

\begin{theorem}\label{asymptotic}
Let $E$ be a Banach space, and let $n< \cd(E)$. Let $F$  be an infinite dimensional
Banach space
with a shrinking metric compact approximation of the identity
$(L_{\beta}) \subset \mathcal K(F)$ such that $\sup_\beta\|L^{\beta}\| \leq 1$
and suppose that $F$ is asymptotically uniformly smooth. Then, $\mathcal
P_w(^nE, F)$ is an \HB-subspace of $\mathcal P(^nE, F)$.
\end{theorem}

\begin{proof}
Consider the projection $\Lambda$, given in \eqref{projection}, under which
$\mathcal P_w(^nE, F)$ is an ideal in $\mathcal P(^nE, F)$. For any $f \in
\mathcal P(^nE, F)^*$, write $f=g+h$ as in \eqref{decomposition}. Then, as
commented in Remark~\ref{converge},  $\|g\|\leq \|f \|$ and $\|h\|\leq \|f \|$.
In order to finish, we have to prove that $\|g\|<\|f\|$, for $h\neq 0$.

Fix $\eps>0$ and take $P \in \mathcal P(^nE, F)$ and $Q \in \mathcal P_w(^nE,
F)$ such that
$$
\|P\|=\|Q\|=1,\quad h(P) > \|h\|-\eps\quad \text{and}\quad g(Q)> \|g\|-\eps.
$$
Since $\lim_{\beta}L_{\beta} \circ Q=Q$, we can choose $\beta_0$ satisfying
$|g(L_{\beta_0}\circ Q )|> \|g\|-2\eps$. Change, if necessary, $Q$ to $\lambda
Q$ (with $|\lambda|=1$) to obtain $g(L_{\beta_0}\circ Q )> \|g\|-2\eps$. For $t
>0$, consider $L_{\beta_0}\circ Q + t L^{\beta}\circ P$ and take a net
$(x_{\beta}) \subset B_E$ with
$$
\overline{\lim}_{\beta} \|L_{\beta_0}\circ Q + t L^{\beta}\circ P\|_{\mathcal
P(^nE, F)} =\overline{\lim}_{\beta} \|(L_{\beta_0} \circ Q)(x_{\beta}) + t
(L^{\beta}\circ P)(x_{\beta})\|_F.
$$
Now, note that $((L^{\beta}\circ P)(x_{\beta}))$ is weakly null. Indeed,
$(P(x_{\beta}))$ is bounded and for any $y^* \in F^*$,
$\lim_{\beta}(L^{\beta})^* y^*=0$. Then,
$$
\lim_{\beta}\langle (L^{\beta}\circ
P)(x_{\beta}), y^*\rangle = \lim_{\beta}\langle P(x_{\beta}), (L^{\beta})^*
y^*\rangle=0
$$
for any $y^* \in F^*$. On the other hand, the compact set
$L_{\beta_0}(B_F)\subset B_F$ contains the net $((L_{\beta_0}\circ
Q)(x_{\beta}))$. Therefore, Lemma~\ref{Delpech} can be applied to get
$$
\overline{\lim}_{\beta} \|L_{\beta_0}\circ Q + t L^{\beta}\circ P\|_{\mathcal
P(^nE, F)}  \leq 1+ \overline{\rho}_F(\overline{\lim}_{\beta}\|t (L^{\beta}\circ
P)(x_{\beta})\|_F) \leq 1 + \overline{\rho}_F(t).
$$
As observed in Remark~\ref{converge}, $\lim_{\beta }g(L^{\beta}\circ P)=0$.
Then, $|g(L^{\beta}\circ P)|<\eps$, for $\beta\geq \beta_1$. Also, $h(L^{\beta}
\circ P) > \|h\|-\eps$, since $L_{\beta} \circ P$ belongs to $\mathcal P_w(^nE,
F)$ and $h(P)=h(L^{\beta} \circ P)$.

Combining the previous estimates,  we conclude for $t>0$ and $\beta\geq \beta_1$ that
$$
\begin{array}{rl}
(\|g\|-2\eps) -t\eps+ t(\|h\|-\eps) < &  g(L_{\beta_0}\circ Q)-t|g(L^{\beta}\circ
P)|+t h(L^{\beta}\circ P)|\\ \leq & |f(L_{\beta_0} \circ Q+t L^{\beta}\circ P)|.
\end{array}
$$
Then, for $t>0$ and $\eps>0$,
$$
\begin{array}{rl}
\|g\|+ t\|h\|- 2\eps (1+t) \le  & \overline{\lim}_{\beta} |f(L_{\beta_0} \circ
Q+t L^{\beta}\circ P)| \\ \leq & \|f\| \,
\overline{\lim}_{\beta}\|L_{\beta_0}\circ Q+tL^{\beta}\circ P\| \\  \leq & \|f\|
(1+\overline{\rho}_F(t)).
\end{array}
$$
Thus, $\|g\|+ t\|h\| \leq  \|f\| (1+\overline{\rho}_F(t))$. Now, suppose that
$\|g\| = \|f\|$, then $t \|h\|\leq \|f\| \overline{\rho}_F(t)$, for $t>0$. Since
$F$ is asymptotically smooth, $\lim_{t \to
0}\displaystyle\frac{\overline{\rho}_F(t)}{t}=0$ and $h=0$, which completes the
proof.
\end{proof}

Our next result gives another set of sufficient conditions, also related with
the notion of smoothness, under which $\mathcal P_w(^n E, F)$ is an \HB-subspace
in $\mathcal P(^n E, F)$. It is reminiscent of \cite[Theorem 1]{Oja-89}. The
proof is similar to that of the above theorem and we omit it.

\begin{theorem}\label{Oja}
Let $E,F$ be Banach spaces, and let $n < \cd(E)$. Suppose that there exists a metric
compact approximation of the identity  $(L_{\beta})\subset \mathcal K(F)$ such
that $\sup_\beta\|L^{\beta}\| \leq 1$ and
for any $\varepsilon >0$ there exist $\mu>0$ and $\beta_0$ so that
\begin{equation}\label{eq:teo Oja}
\sup_{\|y\|, \|z\|\leq 1} \|L_{\beta} y + \mu L^{\beta}z\| \leq 1 +
\varepsilon\mu \, , \text{ for all } \beta\geq \beta_0.
\end{equation}
Then, $\mathcal P_w(^n E, F)$ is an \HB-subspace of $\mathcal P(^n E, F)$.
\end{theorem}

Following \cite[Definition~3.6]{Hen} we say that a Schauder basis of a Banach
space is \textit{uniformly smooth} if for every $\varepsilon>0$ there exists
$\delta>0$ such that $\|x + y\| + \|x-y\| < 2 + \varepsilon \|y\|$, whenever $x$
and $y$ have disjoint supports with respect to the basis, $\|x\| = 1$ and
$\|y\|< \delta$. Note that by using convex combinations, the definition can be
restated for $x$ with $\|x\|\le 1$. The next result should be compared with
\cite[Theorem~3.7]{Hen}.

\begin{corollary}\label{basis}
Let $E,F$ be Banach spaces, and let $n< \cd(E)$. Suppose that $F$ has a uniformly
smooth $1$-unconditional basis. Then $\mathcal P_w(^nE, F)$ is an \HB-subspace
of $\mathcal P(^nE, F)$.
\end{corollary}

\begin{proof}
Let  $\Pi_N$ be the natural projection on $F$ onto the subspace generated by the
first $N$ elements of the $1$-unconditional basis. Then, $(\Pi_N)$ is a metric
compact approximation of the identity and satisfies $\sup_N \|\Pi^N\|\le 1$.
Also, \eqref{eq:teo Oja} in Theorem~\ref{Oja} holds. Indeed, take $\varepsilon
>0$ and consider $\mu=\delta$ as in the definition of uniform smoothness of the
basis. For any $y,z\in B_F$ and $N\in \mathbb N$, being the basis
$1$-unconditional, we have $\|\Pi_N y + \mu \Pi^N z\|\leq 1 + \mu\varepsilon/2$.
Thus, an immediate application of Theorem \ref{Oja} gives the result.
\end{proof}

The above corollary can be applied to show some examples of \HB-subspaces of
polynomials where the spaces $\ell_p$ and the Lorentz sequence spaces $d(w,p)$
appear. Recall that, for $1<p<\infty$, both $\ell_p$ and $d(w,p)$ have uniformly
smooth 1-unconditional bases.  Also, the critical degree of $\ell _p$ is the
integer number satisfying $p\le  \cd(\ell _p) < p+1$ and $\cd(\ell _p, \ell_q)$
is the integer satisfying $\frac{p}{q}\leq \cd(\ell _p, \ell_q) <
\frac{p}{q}+1$. For the case of $\cd(\ell_p, d(w,q))$, a restatement of (I)
and (II) in~\cite[p.~705]{DiLa} reads as
$$
\cd(\ell _p, d(w,q))=\max\{k\in\mathbb N: k<\frac{p}{q}+1 \textrm{ and }
w\not\in\ell_{(\frac{p}{(k-1)q})^*}\}.
$$

\begin{example} \label{HBnoM-ideal} Let $1<p,q<+\infty$.
\item[\rm (a)] Let $E$ be a Banach space, and let $n<\cd(E)$. Then,
\begin{itemize}
\item $\mathcal P_w(^n E, \ell _q)$ is an \HB-subspace of $\mathcal P(^n E, \ell _q)$.
\item $\mathcal P_w(^n E, d(w,q))$ is an \HB-subspace of $\mathcal P(^n E, d(w,q))$.
\end{itemize}

\item[\rm (b)] Let $\cd(\ell _p, \ell_q) < n <\cd(\ell_p)$. Then,
$\mathcal P_w(^n\ell_p, \ell_q)$ is an \HB-subspace but not an $M$-ideal in
$\mathcal P(^n\ell_p, \ell_q)$.

\item[\rm (c)] Let $\cd(\ell _p, d(w,q)) < n <\cd(\ell_p)$. Then, $\mathcal
P_w(^n\ell_p, d(w,q))$ is an \HB-subspace but not an $M$-ideal in $\mathcal
P(^n\ell_p, d(w,q))$. The same result holds for $n=\cd(\ell _p, d(w,q))$ for the
case $\cd(\ell _p, d(w,q))<\frac{p}{q}$.
\end{example}

The statements about not being $M$-ideals in the previous examples are
proved in  \cite[Theorems~3.2 and 3.9]{DiLa}.

Now, we consider conditions satisfied by the domain space $E$ so that we also
have geometric structures in $\mathcal P(^nE;F)$. Similarly to what happens in
Lemma \ref{idealF}, we will describe $\mathcal P _w(^nE, F)$ as an ideal of
$\mathcal P(^nE, F)$ whenever $E^*$ has a metric compact approximation of the
identity with adjoint operators. Here, no restrictions on the degree of the
polynomials are imposed.

\begin{lemma}\label{idealE}
Let $E,F$ be Banach spaces such that $E^*$ has a metric compact approximation of
the identity with adjoint operators. Then,  $\mathcal P_w(^nE, F)$ is an ideal
of $\mathcal P(^nE, F)$ for all $n \in \mathbb N$.
\end{lemma}

\begin{proof}
Let $(K_{\alpha}) \subset \mathcal K(E)$ be a net satisfying
$\lim_{\alpha}K_{\alpha}^*x^*=x^*$, for all $x^* \in E^*$ and $\sup_\alpha
\|K_{\alpha}\| \leq 1$. Without loss of generality, we may assume that
$(K_{\alpha})$ is weak$^*$-convergent in $\mathcal K(E)^{**}$.  Therefore, as in
Lemma \ref{idealF}, the mapping
$$
\Lambda\colon \mathcal P(^nE, F)^* \to
\mathcal P(^nE, F)^*
$$
given by
\begin{equation}\label{projection2}
\Lambda(f)(P)=\lim_{\alpha}f(P\circ K_{\alpha}),
\end{equation}
is well defined. It is clear that $\Lambda$ is linear and $\|\Lambda\|\leq 1$.
By \cite[Lemma~ 2.1]{DiLa}, $\lim_{\alpha}\| P-P\circ  K_{\alpha}\|=0$ for every
$P \in \mathcal P_w(^nE, F)$. Then, $\Lambda(f)(P)=f(P)$, for all $P\in \mathcal
P_w(^nE, F)$. Furthermore, $\Lambda$  is a projection: for all $Q\in \mathcal
P(^nE, F)$, $(Q\circ  K_{\alpha})$ belongs to $\mathcal P_w(^nE, F)$. Thus, for
all $Q \in \mathcal P(^nE, F)$,
$$
\Lambda(\Lambda(f))(Q)=\lim_{\alpha}\Lambda(f)(Q \circ
K_{\alpha})=\lim_{\alpha}f(Q\circ  K_{\alpha})=\Lambda(f)(Q),
$$
and
$\Lambda^2=\Lambda$. It is easy to check that $\ker \Lambda= \mathcal P_w(^nE,
F)^{\perp}$ and the result follows.
\end{proof}

The next result gives a sufficient condition to obtain the dual space
$\mathcal P_w(^nE, F)^*$ as a quotient. We denote by $\pi$ the projective tensor
norm. Recall that $\overline P$ denotes the canonical extension of $P$ in
$\mathcal P(^nE, F)$ to $\mathcal P(^nE^{**}, F^{**})$.

\begin{proposition}\label{quotient}
Let $E,F$ be Banach spaces such that  $\mathcal P_w(^nE, F)$ does not contain
$\ell_1$. Then, the application $j\colon \widehat{\otimes}^{n,s}_{\pi_s} E^{**}
\widehat{\otimes}_{\pi} F^* \to \mathcal P_w(^nE, F)^*$, given on any elementary
tensor $u\otimes y^*$ by $j(u\otimes y^*)(P)=y^*(\overline{P}(u))$, is a
quotient mapping.
\end{proposition}

\begin{proof}
Take $v\in \widehat{\otimes}^{n,s}_{\pi_s} E^{**} \widehat{\otimes}_{\pi} F^*$.
For each representation of $v$ of the form $\sum_i u_i\otimes y^*_i$, with
$u_i\in \widehat{\otimes}^{n,s}_{\pi_s} E^{**}$ and $y^*_i\in  F^*$ for all $i$,
 we have
$$
|j(\sum_i u_i\otimes y^*_i)(P)|= |\sum_i y^*_i(\overline{P}(u_i))| \leq  \|P\|
\sum_i \|u_i\| \|y^*_i \|.
$$
So, $j$ is continuous and $\|j\|=1$. Using Haydon's characterization of spaces
not containing $\ell_1$, we may write the unit ball of $\mathcal P_w(^nE, F)^*$
as the closed convex hull of its extreme points. Now, by
\cite[Proposition~1.2]{DiLa}, with $e_z(P) = \overline P(z)$ for $z\in E^{**}$,
we obtain
$$
\begin{array}{rcl}
B_{\mathcal P_w(^nE, F)^*}&= &\overline{\Gamma(Ext_{\mathcal P_w(^nE, F)^*})}
\subset \overline{\Gamma(e_z \otimes y^*\colon z\in S_{E^{**}}, y^*\in S_{F^*})}
\\
& \subset&\overline{j(B_{\widehat{\otimes}^{n,s}_{\pi_s} E^{**}
\widehat{\otimes}_{\pi} F^*})} \subset B_{\mathcal P_w(^nE, F)^*}.
\end{array}
$$
Then, all the inclusions are (actually) equalities and $j$ is a quotient mapping.
\end{proof}

In the next result we show that the natural hypothesis on $E$ and $F$ guarantee
that $\mathcal P_w(^nE, F)$ does not contain $\ell_1$, and the above proposition
can be applied. We will appeal to the result by Stegall which asserts that if a
Banach space $E$ has a separable subspace whose dual is nonseparable, then $E^*$
lacks the Radon--Nikod\'{y}m property, see for instance \cite[Theorem
VII.2.6]{DU}.

\begin{proposition}\label{nocontainment}
Let $E, F$ be Banach spaces such that $E^{**}$ and $F^*$ have the
Radon--Nikod\'{y}m property. Then, $\mathcal P_w(^nE, F)^*$ has the
Radon--Nikod\'{y}m property and hence $\mathcal P_w(^nE, F)$ does not contain
$\ell_1$ for all $n \in \mathbb N$.
\end{proposition}

\begin{proof}
For any  $P\in \mathcal P_w(^nE, F)$ we consider its associated symmetric
multilinear map $A$ and define $T_P\in \mathcal L(E, \mathcal P_w(^{n-1}E, F))$ as
the operator given by $T_P(x)(\tilde x)=A(x, \tilde x, \dots, \tilde x)$. By
\cite[Theorem~2.9]{AHV}, $T_P$ is a well-defined compact operator and $\|P\|\leq
\|T_P\|\leq e\|P\|$. Then, the mapping $\Phi\colon \mathcal P_w(^n E, F) \to
\mathcal K(E,\mathcal P_w(^{n-1}E, F))$ given by $\Phi(P)=T_P$ is an
isomorphism with its image. As the Radon--Nikod\'{y}m property is preserved by isomorphisms,
induction on $n$ and \cite[Theorem~1.9]{RuSte} yield the result.
\end{proof}

\begin{lemma}\label{Arensregular}
Let $E,F$ be Banach spaces such that  $\mathcal P_w(^nE, F)$ does not contain
$\ell_1$ and such that $n<\cd(E)$. Suppose that  $E^*$ has a
metric compact approximation of the identity with adjoint operators given by
$(K_{\alpha}) \subset \mathcal K(E)$. Then,
$\lim_{\alpha} P\circ K^{\alpha}=0$ in the topology $\sigma(\mathcal P(^nE, F),\mathcal
P_w(^nE, F)^*)$, for any  $P\in \mathcal P(^nE, F)$.
\end{lemma}

\begin{proof}
By Proposition \ref{quotient}, the application $j\colon
\widehat{\otimes}^{n,s}_{\pi_s} E^{**} \widehat{\otimes}_{\pi} F^* \to \mathcal
P_w(^nE, F)^*$, defined by $j(u\otimes y^*)(P)=y^*(\overline{P}(u))$ is a
quotient mapping. We will show for $\tilde{u} =\sum_{i=1}^M  w_i\otimes
w_i\otimes\cdots\otimes w_i \otimes y_i^*$ that $\lim_{\alpha}\langle\tilde{u},
P\circ K^{\alpha}\rangle = 0$. So, as for any $h\in \mathcal P_w(^nE, F)^*$
there exists  $u \in \widehat{\otimes}^{n,s}_{\pi_s} E^{**}
\widehat{\otimes}_{\pi} F^* $ so that $j(u)=h$ and $u$ can be approximated by
such $\tilde{u}$'s, the result follows. Take $\tilde{u}$ as above. Then
$\langle\tilde{u}, P\circ K^{\alpha}\rangle =\sum_{i=1}^M \overline{P\circ
K^{\alpha}}( w_i)(y_i^*) =\sum_{i=1}^M \overline{P}((K^{\alpha})^{**}(
w_i))(y_i^*)$.  As $n<\cd(E)$, $\overline P$ is $w^* -
w^*$ continuous and, since  $\lim_\alpha(K^{\alpha})^{**}(w_i)=0$ in the
$w^*$-topology, we obtain $\lim_{\alpha}\langle\tilde{u}, P\circ
K^{\alpha}\rangle =0$, which completes the proof.
\end{proof}

Proposition~1.4 in~\cite{DiGoJa} provides an appropriate equivalence of
asymptotic uniform convexity of power $p$. With a slight modification of its
proof we drop the hypothesis of separability and obtain a refinement, analogous
to condition (c) in Lemma~\ref{Delpech}, as follows.

\begin{lemma}\label{power-p} Let $E$ be an infinite dimensional Banach space
and let $1<p<\infty$. The following statements are equivalent.
\begin{enumerate}[\upshape (a)]
\item $E$ has modulus of asymptotic uniform convexity of power $p$.
\item There exists a constant $C>0$ such that for every $x \in S_E$ and every
bounded weakly null net  $(w_{\alpha})$ in $E$, we have
$$
\overline{\lim}_{\alpha}\|x+w_{\alpha}\|^p \geq 1 +
C\,\overline{\lim}_{\alpha}\|w_{\alpha}\|^p.
$$
\item There exists a constant $C>0$ such that for every  net $(x_{\alpha})$ in
a compact set of $B_E$ and every bounded weakly null net  $(w_{\alpha})$ in
$E$, we have
$$
\overline{\lim}_{\alpha}\|x_{\alpha}+w_{\alpha}\|^p \geq
\overline{\lim}_{\alpha}  \|x_{\alpha}\|^p + C \, \|w_{\alpha}\|^p.
$$
\end{enumerate}
\end{lemma}

\begin{proposition}\label{useful}
Let $E,F$ be Banach spaces such that $E^{**}$ and  $F^*$ enjoy the
Radon--Nikod\'{y}m property. Suppose that $E$ has modulus of asymptotic uniform
convexity of power $n=\cd(E,F)$, with $n<\cd(E)$, and suppose that $E$ has a
shrinking metric compact approximation of the identity $(K_{\alpha}) \subset
\mathcal K(E)$ satisfying $\sup_\alpha \|K^{\alpha}\| \leq 1$ and\,
$\overline{\lim}_{\alpha}\|Id_E-2K_{\alpha}\| \leq 1$.

Then, there exists $C>0$ such that $\mathcal P_w(^nE, F)$ is an $M(1,C)$-ideal
of $\mathcal P(^nE, F)$. Furthermore, $\mathcal P_w(^nE, F)$ is an \HB-subspace
of $\mathcal P(^nE, F)$.
\end{proposition}

\begin{proof} We proceed as in Theorem \ref{asymptotic}.
Consider $\Lambda$ as in \eqref{projection2}, under which $\mathcal P_w(^nE,
F)$ is an ideal in $\mathcal P(^nE, F)$ and write $f \in \mathcal P(^nE, F)^*$,
$f=g+h$ as in~\eqref{decomposition}, where  $\|g\|=\|\Lambda(f)\|\leq \|f\|$.

Consider $\varepsilon >0$ and take $P \in \mathcal P(^nE, F)$, $Q  \in \mathcal
P_w(^nE, F)$ with
$$
\|P \|=\|Q \|=1,\quad  h(P )\geq \|h\|-\varepsilon\quad \text{and}\quad g(Q
)\geq \|g\|-\varepsilon.
$$
For any $\alpha$, the polynomial $P -P \circ K^{\alpha}$ is weakly continuous on
bounded sets at 0 (see, for instance, the proof of \cite[Proposition 2.2]{Dim}).
Since $n=\cd(E,F)$, the net $(P -P \circ K^{\alpha})_\alpha$  is in $\mathcal
P_w(^nE, F)$. Then, with $h \in \mathcal P_w(^nE, F)^{\perp}$, $h(P )=h(P \circ
K^{\alpha})\geq \|h\| - \varepsilon$, for all $\alpha$. Also, as $\lim_{\alpha}Q
\circ K_{\alpha}=Q $, there exists $\alpha_0$ so that
$|g(Q \circ K_{\alpha})|> \|g\|-2\varepsilon$ and
$\|Id_E-2K_{\alpha_0}\|<1+\varepsilon$, for all $\alpha \geq \alpha_0$. Changing
$Q \circ K_{\alpha_0}$ to $\lambda Q \circ K_{\alpha_0}$ with $|\lambda|=1$, if
necessary, we may assume that $g(Q \circ K_{\alpha_0})> \|g\|-2\varepsilon$.
Now, with $C>0$ (to be fixed later) we have
$$
\begin{array}{rl}
\|f\|\, \|Q \circ K_{\alpha_0}+C P \circ K^{\alpha}\| & \geq |f( Q \circ
K_{\alpha_0}+C P \circ K^{\alpha})|\\
& \geq g(Q \circ K_{\alpha_0})+C h(P \circ K^{\alpha})-C|g(P \circ K^{\alpha})|\\
& \geq \|g\| +  C \|h\| -(2+C) \varepsilon -C|g(P \circ K^{\alpha})|.
\end{array}
$$
By Lemma~\ref{Arensregular}, we may find $\alpha_1 \geq \alpha_0$ so that $|g(P
\circ K^{\alpha})|<\varepsilon$ for $\alpha \geq \alpha_1$. Thus, we obtain
$$
\|f\|\, \|Q \circ K_{\alpha_0}+CP \circ K^{\alpha}\| \geq \|g\|+C
\|h\|-2\varepsilon(1+C), \quad  \text{  for all  } \alpha \geq \alpha_1.
$$
Now, take $(x_{\alpha}) \subset B_E$ such that $\overline{\lim}_{\alpha}\|Q
\circ K_{\alpha_0}+CP \circ K^{\alpha}\|  =  \overline{\lim}_{\alpha}\|Q
(K_{\alpha_0} x_{\alpha})+CP (K^{\alpha} x_{\alpha})\|$ and note that
$(K_{\alpha_0}x_{\alpha})$ is contained in a compact subset of $B_E$ and that
$(K^{\alpha}{x_{\alpha}})$  is weakly null. Since $E$ has modulus of asymptotic
convexity of power $n$ we apply Lemma~\ref{power-p}, with $C>0$ as in item (c),
and get
$$
\begin{array}{rl}
\overline{\lim}_{\alpha}\|Q \circ K_{\alpha_0}+CP \circ K^{\alpha}\| & \le
\overline{\lim}_{\alpha}\|Q \| \|K_{\alpha_0} x_{\alpha}\|^n +C \|P \|
\|K^{\alpha} x_{\alpha}\|^n\\ & = \overline{\lim}_{\alpha}\|K_{\alpha_0}
x_{\alpha}\|^n +C  \|K^{\alpha} x_{\alpha}\|^n \\
& \le  \overline{\lim}_{\alpha}\| K_{\alpha_0} x_{\alpha}+ K^{\alpha}x_{\alpha}
\|^n \\
&  \leq \overline{\lim}_{\alpha}\|K_{\alpha_0} + K^{\alpha}\|^n\\
& \le \|Id_E-2K_{\alpha_0}\|^n,
\end{array}
$$
where the last inequality, being standard, can be found for instance in
\cite[p. 300]{HWW}. Then, we may find $\alpha_2>\alpha_1$ so that
$$
\|Q  K_{\alpha_0}+CP  K^{\alpha_2}\| < 1+ \varepsilon,
$$
and therefore,
$$
\|g\|+C\|h\|-2\varepsilon(1+C) \leq \|f\| (1+\varepsilon).
$$
Since  $\varepsilon>0$ is arbitrary, $\|g\|+C\|h\| \leq \|f\|$ and  $\mathcal
P_w(^nE, F)$ is an $M(1,C)$-ideal in $\mathcal P(^nE, F)$.

To prove  that $\mathcal P_w(^nE, F)$ is also an \HB-subspace of $\mathcal
P(^nE, F)$ note that for
$h \neq 0$, $\|g\|<\|g\|+C\|h\| \leq \|f\|$. On the other hand,  for $\alpha
>\alpha_1$,
$$
\|f\| \geq |f(P \circ K^{\alpha})|\geq h(P \circ K^{\alpha})-|g(P \circ
K^{\alpha})|\geq \|h\|-2\varepsilon,
$$
implying that $\|f\|\geq \|h\|$, which completes the proof.
\end{proof}

\section{Ideal structures inherited by the range space}

Our purpose in this section is to give sufficient conditions on the spaces $E$
and $F$ under which those geometric properties enjoyed by  $\mathcal P_w(^nE,
F)$ as a subspace of $\mathcal P(^nE, F)$ are inherited by the range space $F$
as a subspace of $F^{**}$. We start with \HB-smoothness presenting an extension
to the polynomial setting of  \cite[Theorem~7]{Oja-Pol}. Our proof also follows
their ideas.

\begin{proposition}\label{surjection}
Let $E,F$ be Banach spaces such that there
exists a surjection  $\rho\colon E \to F$. If $\mathcal P_w(^nE, F)$ is
\HB-smooth in $\mathcal P(^nE, F)$ for some $n \in \mathbb N$, then $F$ is
\HB-smooth in $F^{**}$.
\end{proposition}

\begin{proof}
Denote by $NA(E^{**})$ the subset of norm-attaining elements in
$E^{**}$ and consider
$$
A:=\{ \rho^{**}(x^{**})\colon x^{**}\neq 0,\quad x^{**}\in NA(E^{**})\}.
$$
As $\rho$ is surjective, $\rho^{**}$ is also surjective and, by the
Bishop-Phelps theorem, $\overline{A}=F^{**}$. By~\cite[Theorem~1~(c)]{Oja-Pol}
it is enough to show that for every  $\rho^{**}(x^{**}) \in A$ and any sequence
$(y_k) \subset F$ with $\|y_1\|< 1$, $\|y_{k+1}-y_k\|< 1$ there are $y \in F$
and $k_0 \in \mathbb N$ so that
$\|\rho^{**}x^{**}- y\pm y_{k_0}\|< k_0$.

Fix $x_0^{**}\ne 0$ in $NA(E^{**})$, take $x_0^* \in S_{E^*}$ such that
$x_0^{**}(x_0^*)=\|x_0^{**}\|$ and define, for $k\in \mathbb N$, the
$n$-homogeneous polynomial $P_k(x)=x_0^{*}(x)^n y_k \in \mathcal P_w(^nE, F)$.
It is clear that
$$
\|P_1\|<1\quad \text{and}\quad \|P_{k+1}-P_k\| \leq \|y_{k+1}-y_k\|< 1,\quad
\text{for all}\ k \in \mathbb N.
$$
Now, define $Q(x)=\rho(x)\,x_0^*(x)^{n-1}\|x_0^{**}\|$ and consider its
Aron--Berner extension $\overline{Q}$ given by $\overline{Q}(x^{**})=
\rho^{**}(x^{**})\,x^{**}(x_0^{*})^{n-1}\|x_0^{**}\|$.

By~\cite[Theorem~1~(a)]{Oja-Pol} there exist $R\in \mathcal P_w(^nE, F)$ and
$k_0 \in \mathbb N$ with $\|Q-R\pm P_{k_0}\|< k_0$. As the Aron--Berner
extension preserves the norm, we obtain
$\|\overline{Q}(x_0^{**})-\overline{R}(x_0^{**})\pm
\overline{P}_{k_0}(x_0^{**})\|< k_0 \|x_0^{**}\|^n$.
The result follows by taking
$y=\displaystyle\frac{\overline{R}(x_0^{**})}{\|x_0^{**}\|^n}$.
\end{proof}

As an immediate consequence we have the following result.

\begin{corollary}\label{F=E}
Let $E$ be a Banach space. If $\mathcal P_w(^nE, E)$ is \HB-smooth in $\mathcal
P(^nE, E)$ for some $n \in \mathbb N$, then $E$ is $ \textit{HB}$-smooth in
$E^{**}$.
\end{corollary}

Note that the above corollary says that $\mathcal P_w(^n\ell_1, \ell_1)$ is
not \HB-smooth in $\mathcal P(^n\ell_1, \ell_1)$ for any $n \in \mathbb N$.

Now we address the notion of \HB-subspace, when the range space $F$ is a
quotient of the space $E$. The following technical result,  inspired by
\cite[Proposition 2.3]{W}, will be useful.

\begin{lemma}\label{HB-smooth}
Let $J$ be an ideal in the Banach space $E$ under the projection $\q$. Suppose
that $J$ is \HB-smooth and $\lambda \in (0,2]$. The following statements are
equivalent.
\begin{enumerate}[\upshape (i)]
\item  $\|Id_{E^*} - \lambda \q\| \leq 1$.
\item  For each $x \in B_E$ there exists a net $(y_{\alpha})\subset J$ such
that $\lim_{\alpha} y_{\alpha}= x$ in the $\sigma(E,J^*)$-topology and
$\overline{\lim}_{\alpha} \|x-\lambda y_{\alpha}\|\leq 1$.
\item  For each $x \in B_E$ and $\varepsilon >0$ there exists a net
$(y_{\alpha})\subset J$ such that $\lim_{\alpha} y_{\alpha}= x$ in the
$\sigma(E,J^*)$-topology and $\overline{\lim}_{\alpha} \|x-\lambda
y_{\alpha}\|\leq 1 + \varepsilon$.
\end{enumerate}
\end{lemma}

\begin{proof} To prove that (i) implies (ii), consider the set of indices
$A:=\{\alpha=(N,M,\eps)\colon N\in \text{FIN}(E^{**}),\, M\in
\text{FIN}(E^{*}),\, \eps >0\}$, where $\text{FIN}$ denotes the set of all
finite dimensional subspaces, with the usual order. By the Principle of Local
Reflexivity, for any $\alpha \in A$ there exists $T_{\alpha}\colon N \to E$
so that
\begin{itemize}
\item $\langle T_{\alpha}x^{**}, x^*\rangle = \langle x^{**}, x^*\rangle$, for
$x^{**} \in N,\, x^* \in M$,
\item $\|T_{\alpha}\|\leq 1 + \eps$,
\item $T_{\alpha}|_{N\cap E} = Id_E$.
\end{itemize}
Fix $x \in B_E$ and consider $y_{\alpha}= T_{\alpha}(\q^*x)\in E$, defined for
$\alpha$ large enough. Fix $y^*\in J^*$ and $\eps >0$, then if $\alpha \ge
(\{\q^* x\}, \{y^* \}, \eps)$, as $J^*$ is the range of $\q$, we have
$$
\langle y^*,y_{\alpha}\rangle= \langle y^*,T_{\alpha}(\q^*x)\rangle = \langle
y^*,\q^* x\rangle = \langle \q y^*, x\rangle = \langle y^*,x\rangle,
$$
and $\lim_{\alpha} y_{\alpha}= x$ in the $\sigma(E,J^*)$-topology. On the other
hand, also for $\alpha$ large enough,
$$
\begin{array}{rl}
\|x-\lambda y_{\alpha}\|&=
\|T_{\alpha}(x-\lambda \q^* x)\| \leq (1+\eps) \|x-\lambda \q^*x\| \leq \\
& \leq(1+\eps)  \|Id_{E^{**}}-\lambda \q^*\| = (1+\eps)  \|Id_{E^*}-\lambda
\q\|\leq 1+\eps.
\end{array}
$$
Then, $\overline{\lim}_{\alpha} \|x-\lambda y_{\alpha}\|\leq 1$ and (ii)
follows.

Clearly, (ii) implies (iii). To prove that (iii) implies (i), fix
$\varepsilon>0$, $x \in B_E$ and  choose  $(y_{\alpha})\subset J$ satisfying
(iii). For each $x^* \in  B_{E^*}$, as $\lim_{\alpha}
x^*(y_{\alpha})=\lim_{\alpha} \q x^*(y_{\alpha})= \q x^*(x)$, we have that
$$
|x^*(x)-\lambda \q x^*(x)| =
\lim_{\alpha} |x^*(x)-\lambda x^*(y_{\alpha})|\leq \|x^*\| \, \,
\overline{\lim}_{\alpha} \| x-\lambda y_{\alpha} \| \leq 1 +\varepsilon.
$$
Since $\varepsilon$ is arbitrary, the implication follows.
\end{proof}

\begin{proposition}\label{quotientHB}
Let $E,F$ be Banach spaces such that there exists a quotient mapping
$\rho\colon E \to F$. If $\mathcal P_w(^nE, F)$ is an \HB-subspace of $\mathcal
P(^nE, F)$, for some $n \in \mathbb N$, then $F$ is an \HB-subspace of $F^{**}$.
\end{proposition}

\begin{proof}
First, note that $F$ is an ideal in its bidual and call $\q$ the associated
projection. By Proposition \ref{surjection}, $F$ is \HB-smooth in  $F^{**}$.
Then, by \cite[Theorem]{Oja-88}, $F$ is strongly \HB-smooth in $F^{**}$ and  we
only have to show that $\|f- \q f\|\leq \|f\|$ for all $f \in F^{***}$. In order
to do so,  we will see that condition  (iii) in Lemma~\ref{HB-smooth} is
satisfied for $\lambda=1$.

Take  $y^{**} \in B_{F^{**}}$ and $\varepsilon >0$. Choose $w^* \in S_{F^*}$ so
that $y^{**}(w^*)>0$ and
$\left(\displaystyle\frac{\|y^{**}\|}{y^{**}(w^*)}\right)^n < 1+\varepsilon$.
Now, with $\mu= y^{**}(w^*)$ define $P \in \mathcal P(^nE, F)$ by
$P(x)=\mu\rho(x)(\rho^*w^*)^{n-1}(x),$ for each $x \in E$.
By Lemma~\ref{HB-smooth}, there exists a net $(P_{\alpha}) \subset \mathcal
P_w(^nE, F)$ converging to $P$ in the $\sigma(\mathcal P(^nE, F),\mathcal
P_w(^nE, F)^*)$-topology and such that $\overline{\lim}_{\alpha}\|P-P_{\alpha}\| \leq 1$.
As $\rho$ is a quotient mapping, $\rho^*$ is an isometry and we may find
$x^{**}\in E^{**}$ with $\rho^{**}(x^{**})=y^{**}$ and $\|x^{**}\| =
\|y^{**}\|$. Since the Aron--Berner extension of each $P_{\alpha}$ has range in
$F$, we define $y_{\alpha}=\overline{P}_{\alpha}(x^{**}/\mu)\in F$, for all
$\alpha$.

Note that each $x^{**}\otimes y^*\in E^{**}\otimes F^*$ acts in a natural way
as an element of $\mathcal P(^nE, F)^*$ and, therefore, as an element  of
$\mathcal P_w(^nE, F)^*$. Then, as $\overline{P}(z)=\mu \rho^{**}(z)
z^{n-1}(\rho^*w^*)$, for $z\in E^{**}$,  we have
$\overline{P}(x^{**}/\mu)=y^{**}$ and
$$
y^{**}(y^*)=(\overline{P}(x^{**}/\mu))(y^*)= \lim_{\alpha}
y^*(\overline{P_{\alpha}}(x^{**}/\mu))=\lim_{\alpha} y^*(y_{\alpha}).
$$
Thus, $y_\alpha\to y^{**}$ in the $w^*$-topology. Also,
$$
\begin{array}{rl}
\overline{\lim}_{\alpha} \|y^{**}-y_{\alpha}\|= &
\overline{\lim}_{\alpha} \|\overline P(x^{**}/\mu)-\overline
P_{\alpha}(x^{**}/\mu)\| \\ \le & \overline{\lim}_{\alpha} \|P - P_\alpha\|\|x^{**}/
\mu\|^n \\ \le & \left(\|y^{**}\|/\mu \right)^n< 1+\varepsilon.
\end{array}
$$
Another application of Lemma~\ref{HB-smooth} gives that $\|Id_{F^{***}}-
\q\|\leq 1$ and, therefore, $F$ is an \HB-subspace of $F^{**}$.
\end{proof}

Finally, we  focus on $M(1,C)$-ideal structures. Recall that  if $J$ is an
ideal in a Banach space $E$ satisfying the $M(1,C)$-inequality the following
condition holds: For any $m \in \mathbb N$, $y_1,y_2, \dots , y_m \in B_J$,
$x \in B_E$, and $\varepsilon >0$, there is $z\in J$ such that $\|y_i+Cx-z\|
\leq 1 + \varepsilon$, for $1 \leq i \leq m$, \cite[Lemma~2.2]{CabNieOja}. In
fact, when dealing with $E=J^{**}$, it is true that being an $M(1,C)$-ideal is
equivalent to an appropriate $2$-ball property of type $(1,C)$. Namely, we have
the following equivalence, which can be proved with the arguments appearing in
the proof of \cite[Lemma 2.3]{CabNieOja}.

\begin{lemma}\label{CabNieOja}
Let $E$ be a Banach space, and let$C\in (0,1]$. The following statements are equivalent.
\begin{enumerate}[\upshape (i)]
\item  $E$ is an $M(1,C)$-ideal of $E^{**}$.
\item  For all $x \in S_E$, $x^{**} \in S_{E^{**}}$ and $\eps >0$, there exists
$x_0\in E$ with $\|\pm x + C x^{**} -x_0\| < 1+ \eps $.
\end{enumerate}
\end{lemma}

We use the above characterization to give an analogous statement  to
Proposition~\ref{quotientHB} in the case of $M(1,C)$-ideals.

\begin{proposition}\label{quotientMC}
Let $E,F$ be Banach spaces such that there exists a quotient mapping
$\rho\colon E \to F$. If $\mathcal P_w(^nE, F)$ is an $M(1,C)$-ideal of
$\mathcal P(^nE, F)$, for some $n \in \mathbb N$ and $C>0$, then $F$ is an
$M(1,C)$-ideal of $F^{**}$.
\end{proposition}

\begin{proof}
Let us prove that condition (ii) in Lemma~\ref{CabNieOja} is satisfied. Fix $y
\in S_F$, $y^{**} \in S_{F^{**}}$ and $\varepsilon >0$. Choose $\delta >0$ such
that $\delta +\displaystyle\frac{1+\delta}{(1-\delta)^{n-1}} < 1+ \varepsilon$
and $y^* \in S_{F^*}$ with $y^{**}(y^*)\geq 1-\delta$. Define, for $x \in E$, $P
\in \mathcal P(^nE, F)$ and $Q \in \mathcal P_w(^nE, F)$ by
$$
P(x)=(\rho^*y^*)^{n-1}(x)\rho(x)\quad \text{and}\quad Q(x)=(\rho^*y^*)^{n}(x) y,
$$
with $\|P\|, \|Q\| \leq \|\rho\|^n = 1$. Due to \cite[Lemma 2.2]{CabNieOja},
there exists $R \in \mathcal P_w(^nE, F)$ so that $\|\pm Q + C P -R \| \leq 1 +
\delta$. As $\rho^{*}$ is an isometry, there is $x^{**} \in S_{X^{**}}$ with
$\rho^{**}(x^{**})=y^{**}$. Extending by Aron--Berner, for $z\in E^{**}$,
$$
\overline{P}(z)=z(\rho^*y^*)^{n-1} \rho^{**}(z) , \quad \text{and}\quad
\overline{Q}(z)=z(\rho^*y^*)^n\, y.
$$
Then, with $\mu= y^{**}(y^*)$, $\overline{P}(x^{**})= \mu^{n-1} y^{**} $ and
$\overline{Q}(x^{**})= \mu^n y$,
$$
\| \pm \mu^n \, y + C \mu^{n-1} \, y^{**} -\overline{R}(x^{**})\|= \|\pm
\overline{Q}(x^{**})+C\overline{P}(x^{**}) -\overline{R}(x^{**})\| \leq
1+\delta.
$$
As $R\in \mathcal P_w(^nE, F)$, the range of $\overline{R}$ is also in $F$ and
we may take $y_0 = \overline{R}(x^{**})/\mu^{n-1}$. Thus,
$$
\|\pm \mu y + C  y^{**}  - y_0\| \leq \frac{1+\delta}{(1-\delta)^{n-1}}.
$$
Finally, $\| \pm y + C y^{**} -y_0\| \leq \| y - \mu y\| + \| \pm \mu  y + C
y^{**} -y_0\| < 1+ \varepsilon$,  and the result follows.
\end{proof}

In the above proposition, the case $C=1$ corresponds to the structure of an
$M$-ideal. The corollary follows directly and seems to be new in this context.

\begin{corollary}\label{quotientMideal}
Let $E,F$ be Banach spaces such that there exists a quotient mapping
$\rho\colon E \to F$. If $\mathcal P_w(^nE, F)$ is an $M$-ideal of $\mathcal
P(^nE, F)$, for some $n \in \mathbb N$, then $F$ is an $M$-ideal of $F^{**}$.
\end{corollary}

As we have already noted in Corollary~\ref{F=E}, it is now immediate to
derive versions of  Proposition~\ref{quotientHB}, Proposition~\ref{quotientMC}
and Corollary~\ref{quotientMideal} for the case $E=F$.
\\
\\
{\bf Acknowledgement.} The work of this paper was initiated while the third author was visiting the
Departamento de Matem\'{a}tica, Universidad de San Andr\'{e}s. She is greatly
indebted to Profs.\ Dimant and Lassalle for the invitation during her sabbatical,
and for the warm hospitality extended by them and their families.
Work partially supported by PAI-UdeSA 2013, CONICET-PIP 0624, ANPCyT PICT 1456,
UBACyT X038 and MINECO grant MTM 2012-3431.

\bibliographystyle{amsplain}

\begin{thebibliography}{99}
\bibitem{AF}
R. Alencar, K. Floret. \textit{Weak-strong continuity of multilinear mappings
and the Pelczynski-Pitt theorem}. J. Math. Anal. Appl. {\bf 206} (1997), no. 2,
532--546.

\bibitem{AB}
R. Aron, P. Berner. {\it A Hahn--Banach extension
theorem for analytic mappings}. Bull. Math. Soc. France  {\bf 106} (1978), 3--24.


\bibitem{AHV}
R. Aron, C. Herv\'es, M. Valdivia. {\it Weakly
continuous mappings on Banach spaces}. J. Funct. Anal. {\bf 52} (1983), 189--204.

\bibitem{AP}
R. Aron, J. B. Prolla. {\it Polynomial approximation of differentiable
functions on Banach spaces}. J. Reine Angew. Math. {\bf 313} (1980), 195--216.

\bibitem{BR-98}
C. Boyd, R. Ryan. {\it Bounded weak continuity of
homogeneous polynomials at the origin}. Arch. Math. (Basel) {\bf 71} (1998),
no. 3, 211--218.

\bibitem{CabNie2}
J. C. Cabello, E. Nieto. {\it On properties of $M$-ideals}.
Rocky Mountain J. Math. {\bf 28} (1998), no. 1, 61–-93.

\bibitem{CabNieOja}
J. C. Cabello, E. Nieto, E. Oja. {\it On ideals of compact operators satisfying
the $M(r,s)$-inequality}. J. Math. Anal. Appl.  {\bf 220} (1998), 334--348.

\bibitem{CaLa}
D. Carando, S. Lassalle.
{\it $E'$ and its relation with vector-valued functions on $E$}. Ark. Mat. {\bf
42} (2004), no. 2, 283--300.

\bibitem{DaGa}
A. Davie, T. Gamelin.  {\it A theorem on polynomial-star approximation}.  Proc.
Amer. Math. Soc. {\bf 106} (1989), no. 2, 351--356.

\bibitem{Delp}
S. Delpech. Approximations h\"{o}ld\'{e}riennes de fonctions entre espaces
d'Orlicz. Modules asymptotiques uniformes. Doctoral thesis. Universit\'{e} de
Bordeaux, 2005.

\bibitem{DU}
J. Diestel, J. J. Uhl. Vector measures. Mathematical Surveys {\bf 15}. American
Mathematical Society, Providence, R. I., 1977.

\bibitem{Dim}
V. Dimant. {\it $M$-ideals of homogeneous polynomials}. Studia Math. {\bf 202}
(2011), 81--104.

\bibitem{DiGoJa}
V. Dimant, R. Gonzalo, J. A. Jaramillo. {\it Asymptotic structure,
$\ell_p$-estimates of sequences, and compactness of multilinear mappings}. J.
Math. Anal. Appl.  {\bf 359} (2009), 680--693.

\bibitem{DiLa}
V. Dimant, S. Lassalle. \textit{$M$-structures in vector-valued polynomial
spaces}.  J. Convex Anal.  {\bf 19} (2012), no. 3, 685--711.

\bibitem{Di}
S. Dineen. Complex analysis on infinite-dimensional spaces.
Springer Monographs in Mathematics. Springer-Verlag London, Ltd.,
London, 1999.

\bibitem{GG}
M. Gonz\'alez, J. Guti\'errez. \textit{The polynomial property (V)}. Arch.
Math. (Basel) {\bf 75} (2000), no. 4, 299--306.

\bibitem{GJ}
R. Gonzalo, J. A.  Jaramillo. \textit{Compact polynomials between
Banach spaces}. Proc. Roy. Irish Acad. Sect. A {\bf 95} (1995), no. 2,
213--226.

\bibitem{HWW}
P. Harmand, D. Werner, W. Werner. $M$-ideals in Banach
spaces and Banach algebras. Lecture Notes in Mathematics, 1547.
Springer-Verlag, Berlin, 1993.

\bibitem{Hen}
J. Hennefeld. {\it $M$-ideals, HB-subspaces, and compact operators}.
Indiana Univ. Math. J. {\bf 28} (1979), no. 6, 927--934.

\bibitem{John}
J. Johnson. {\it Remarks on Banach spaces of compact operators}. J. Funct. Anal.
{\bf 32} (1979), 304--311.

\bibitem{LL}
V. Lima, {\AA}. Lima. {\it Ideals of operators and the metric
approximation property}. J. Funct. Anal. {\bf 210} (2004), no. 1, 148--170.

\bibitem{Oja-88}
E. Oja. {\it Strong uniqueness of the extension of linear
continuous functionals according to the Hahn--Banach theorem}.  (Russian) Mat.
Zametki {\bf 43} (1988), no. 2, 237--246, 302;
translation in Math. Notes {\bf 43} (1988), no. 1--2, 134--139.

\bibitem{Oja-89}
E. Oja. {\it Isometric properties of the subspace of compact operators in the
space of continuous linear operators}. (Russian) Mat. Zametki {\bf 45} (1989),
no. 6, 61--65, 111; translation in Math. Notes {\bf 45} (1989), no. 5--6,
472--475.

\bibitem{Oja-Pol}
E. Oja, M. P\~{o}ldvere. {\it On subspaces of Banach spaces
where every functional has a unique norm-preserving extension}. Studia Math.
{\bf 117} (1996), no. 3, 289--306.

\bibitem{RuSte}
W. M. Ruess, C. P. Stegall. {\it Extreme points in duals of
operator spaces}. Math. Ann. {\bf 261} (1982), no 4, 535--546.

\bibitem{W}
D. Werner. {\it $M$-ideals and the ``basic inequality''}. J. Approx.
Theory {\bf 76} (1994), no. 1, 21--30.
\end{thebibliography}

\end{document}